\documentclass[11pt]{article}
\usepackage{amsfonts, amsmath, amssymb, latexsym, eucal, amscd} 
\usepackage{cite}
\usepackage[all]{xy}
\newtheorem{theorem}{Theorem}[section]
\newtheorem{lemma}[theorem]{Lemma}

\newtheorem{prop}[theorem]{Proposition}
\newtheorem{corollary}[theorem]{Corollary}
\newtheorem{exAux}[theorem]{Example}
\newenvironment{example}{\begin{exAux} \rm}{\end{exAux}}
\newtheorem{Def}[theorem]{Definition}
\newenvironment{defi}{\begin{Def} \rm}{\end{Def}}
\newtheorem{Note}[theorem]{Note}
 
\newtheorem{Problem}[theorem]{Problem}

\newtheorem{Rem}[theorem]{Remark}

\newtheorem{Not}[theorem]{Notation}

\newtheorem{Conj}[theorem]{Conjecture}

\newtheorem{Ass}[theorem]{Assumption}

\newenvironment{proof}{\medskip\noindent{\bf Proof.\ }}{\qed\medskip}
\newenvironment{proofof}[1]{\medskip\noindent{\bf Proof  of {#1}.\ 
}}{\qed\medskip}
\newcommand{\qed}{\hfill\mbox{$\Box$\qquad\qquad}}

\newcommand{\F}{{\mathbb F}}
\newcommand{\C}{\mathbb{C}}
\newcommand{\R}{\mathbb{R}}

\newcommand{\Z}{\mathbb{Z}}

\newcommand{\A}{{A}}
\newcommand{\B}{ {B}}
\newcommand{\J}{{J}}
\newcommand{\I}{{I}}
\newcommand{\E}{{E}}

\newcommand{\q}{{\sf q}}

\renewcommand{\P}{{\sf P}}
\newcommand{\Q}{{\sf Q}}

\newcommand{\Mat}{\text{\rm Mat}}

\renewcommand{\th}{\theta}

\newcommand{\bilin}[1]{\langle #1 \rangle }
\newcommand{\be}[1]{(-1)^{\bilin{#1}}}
\newif\ifDRAFT

\begin{document}

\title{Formal self-duality and numerical self-duality for 
\\
symmetric association schemes}

\author{Kazumasa Nomura and Paul Terwilliger}

\maketitle

\begin{center}
\bf \small Abstract
\end{center}
\begin{quote}  \small
Let ${\mathcal X} = (X, \{R_i\}_{i=0}^d)$ denote a symmetric association scheme.
Fix an ordering $\{E_i\}_{i=0}^d$ of the primitive idempotents of $\mathcal{X}$,
and let $P$ (resp.\ $Q$) denote the corresponding first eigenmatrix (resp.\ second eigenmatrix) of $\mathcal X$.
The scheme $\mathcal X$ is said to be formally self-dual  (with respect to the ordering  $\{E_i\}_{i=0}^d$)
whenever $P=Q$.
We define $\mathcal X$ to be  numerically self-dual (with respect to the ordering  $\{E_i\}_{i=0}^d$)
whenever the intersection numbers and Krein parameters satisfy  $p^h_{i,j} =q^h_{i,j}$ for  $0 \leq h,i,j \leq d$.
It is known that  with respect to the ordering  $\{E_i\}_{i=0}^d$,
formal self-duality implies  numerical self-duality.
This raises the following question:
is it possible that  with respect to the ordering  $\{E_i\}_{i=0}^d$,
$\mathcal X$ is numerically self-dual but not formally self-dual?
This is possible as we will show.
We display an example of a symmetric association scheme and an ordering the primitive idempotents
with respect to which the scheme is numerically self-dual but not formally self-dual.
We have the following additional results about self-duality.
Assume that $\mathcal X$ is $P$-polynomial.
We show that the following are equivalent:
(i)
$\mathcal X$ is formally self-dual with respect to the ordering $\{E_i\}_{i=0}^d$;
(ii)
 $\mathcal X$ is numerically self-dual with respect to the ordering $\{E_i\}_{i=0}^d$.
Assume that the ordering $\{E_i\}_{i=0}^d$ is $Q$-polynomial.
We show that the following are equivalent:
(i)
$\mathcal X$ is formally self-dual with respect to the ordering $\{E_i\}_{i=0}^d$;
(ii)
 $\mathcal X$ is numerically self-dual with respect to the ordering $\{E_i\}_{i=0}^d$.
\end{quote}

\section{Introduction}

In this paper we consider two kinds of self-duality for symmetric association schemes:
formal self-duality and numerical self-duality.
The notion of formal self-duality is well-known  \cite[p.\ 49]{BCN}, \cite[Section 6]{MT}.
We introduce the notion of numerical self-duality
in this paper.
We now recall the notion of formal self-duality.
Consider a symmetric association scheme ${\mathcal X}  = (X, \{R_i\}_{i=0}^d)$.
The associate matrices $\{A_i\}_{i=0}^d$ of $\mathcal X$ form a basis of a subalgebra
$\mathcal M$ of $\Mat_X(\C)$, called the Bose-Mesner algebra of $\mathcal X$.
The algebra $\mathcal M$ has another basis $\{E_i\}_{i=0}^d$, consisting of the primitive idempotents. 
There are matrices $P$ and $Q$ in $\Mat_{d+1}(\C)$ called the first and second eigenmatrices
of $\mathcal X$. 
The matrix $P$ is the transition matrix from the basis $\{E_i\}_{i=0}^d$  to the basis $\{A_i\}_{i=0}^d$,
and the matrix $|X|^{-1} Q$ is the transition matrix from the basis $\{A_i\}_{i=0}^d$ to the basis
$\{E_i\}_{i=0}^d$.
The scheme $\mathcal X$ is said to be formally self-dual (with respect to the given ordering $\{E_i\}_{i=0}^d$
of the primitive idempotents) whenever $P=Q$.
There are many symmetric association schemes that are formally self-dual.
One familiar example is the Hamming scheme \cite[Section III.3.2]{BI}.
In the next two paragraphs we give some additional examples.

The notion of a spin model was introduced by V. F. R.  Jones \cite{Jones} to construct invariants
for knots and links.
A spin model $W$  is a symmetric matrix that satisfies two conditions, called type II and
type III.
In \cite{N:spin} the first author constructed an algebra $N(W)$ that contains $W$ and
is the Bose-Mesner algebra of a symmetric association scheme $\mathcal X$.
In \cite{JMN}, it is shown that $\mathcal X$ is formally self-dual.
The formal self-duality of $\mathcal X$ was used to investigate the
structure of $N(W)$ in  \cite{JMN, NT:spin, CN}.

A symmetric association acheme ${\mathcal X} = (X, \{R_i\}_{i=0}^d)$ is said to be $P$-polynomial whenever
the intersection numbers satisfy the following conditions for $0 \leq h,i,j \leq d$:
\begin{itemize}
\item[(i)]
$p^h_{i,j}=0$ if one of $h,i,j$ is greater than the sum of the other two;
\item[(ii)]
$p^h_{i,j} \neq 0$ if one of $h,i,j$ is equal to the sum of the other two.
\end{itemize}
There is a type of $P$-polynomial scheme, said to have classical parameters $(d,b,\alpha, \beta)$ (see \cite[Section 6.1]{BCN}).
By \cite[Corollary 8.4.4]{BCN}, if $\alpha=b-1$ then the scheme is formally self-dual.
The following schemes satisfy the condition $\alpha=b-1$ and $b \neq 1$:
bilinear forms scheme \cite[p.\ 280]{BCN};
alternating forms scheme \cite[p.\ 282]{BCN};
quadratic forms scheme \cite[p.\ 290]{BCN};
Hermitian forms scheme \cite[p.\ 285]{BCN};
affine $E_6(q)$ scheme \cite[p.\ 340]{BCN};
extended ternary Golay code scheme \cite[p.\ 356]{BCN}.
In \cite{IT:tetra} it is shown  that if $\alpha = b-1$ and $b \neq 1$
then the quantum algebra $U_q(\widehat{\mathfrak{sl}_2})$ acts on the standard module of the scheme,
where $b = q^2$.

We now introduce the concept of numerical self-duality.
For the rest of this section,
let ${\mathcal X} = (X, \{R_i\}_{i=0}^d)$ denote a symmetric association scheme.
The Bose-Mesner algebra of $\mathcal X$ is  closed under  the 
entry-wise product $\circ$.  
The Krein parameters $q^h_{i,j}$ $(0 \leq h,i,j \leq d)$ are defined by
$E_i \circ E_j = |X|^{-1} \sum_{h=0}^d q^h_{i,j}E_h$ for $0 \leq i,j \leq d$.
The scheme $\mathcal X$ is said to be numerically self-dual
(with respect to the given ordering $\{E_i\}_{i=0}^d$ of the primitive idempotents)
whenever 
\[
p^h_{i,j} = q^h_{i,j} \qquad \qquad (0 \leq h,i,j \leq d).
\]
It is known \cite[Section 2.3]{BCN} that if $\mathcal X$ is formally self-dual with respect to 
the ordering $\{E_i\}_{i=0}^d$
then $\mathcal X$ is numerically self-dual with respect to the ordering $\{E_i\}_{i=0}^d$.
For a given ordering $\{E_i\}_{i=0}^d$ of the primitive idempotents,
is it possible that $\mathcal X$ is numerically self-dual but not formally self-dual?
This is possible as we will show.
We will demonstrate with an example of a symmetric  association scheme 
and an ordering of the primitive idempotents
with respect to which the scheme is numerically self-dual but not formally self-dual.

We obtain the following additional results about self-duality.
Assume that ${\mathcal X}$ is $P$-polynomial.
Fix an ordering $\{E_i\}_{i=0}^d$ of the primitive idempotents of $\mathcal X$.
We show that the following are equivalent:
\begin{itemize}
\item[(i)]
 $\mathcal X$ is formally self-dual with respect to the ordering $\{E_i\}_{i=0}^d$;
\item[(ii)]
 $\mathcal X$ is numerically self-dual  with respect to the ordering $\{E_i\}_{i=0}^d$.
\end{itemize}
An ordering $\{E_i\}_{i=0}^d$ of the primitive idempotents of $\mathcal X$ is said to be
$Q$-polynomial whenever the Krein parameters satisfy the following conditions for $0 \leq h,i,j \leq d$:
\begin{itemize}
\item[(i)]
$q^h_{i,j}=0$ if one of $h,i,j$ is greater than the sum of the other two;
\item[(ii)]
$q^h_{i,j} \neq 0$ if one of $h,i,j$ is equal to the sum of the other two.
\end{itemize}
Assume that there is a $Q$-polynomial ordering
$\{E_i\}_{i=0}^d$ of the primitive idempotents of $\mathcal X$.
We show that the following are equivalent:
\begin{itemize}
\item[(i)]
$\mathcal X$ is formally self-dual with respect to the ordering $\{E_i\}_{i=0}^d$;
\item[(ii)]
$\mathcal X$ is numerically self-dual  with respect to the ordering $\{E_i\}_{i=0}^d$.
\end{itemize}

The paper is organized as follows.
In Section \ref{sec:pre} we fix our notation.
In Section \ref{sec:scheme} we recall the  definition of a symmetric association scheme,
and give some formulas.
In Section \ref{sec:dual} we recall the notion of formal self-duality,
and introduce the notion of numerical self-duality.
In Section  \ref{sec:ex} we display an example of a symmetric association scheme 
$\mathcal X$ and an ordering
of the primitive idempotents with respect to which $\mathcal X$ is numerically
self-dual but not formally self-dual.
In Section \ref{sec:Ppoly} we describe our results concerning $P$-polynomial schemes
and $Q$-polynomial schemes.
In Section \ref{sec:duality} we recall the concept of Askey-Wilson duality.
In Section \ref{sec:proof} we use Askey-Wilson duality to  prove our results stated in Section \ref{sec:Ppoly}.

\section{Preliminaries}
\label{sec:pre}

We now begin our formal argument.
Throughout the paper, the following notation is in effect.
Let  $\F$ denote a field. 
For an integer $n \geq 1$, let $\Mat_n(\F)$ denote  the $\F$-algebra consisting of the $n \times n$ matrices
with all entries in $\F$.
Let $V$ denote a vector space over $\F$ with dimension $n$.
Let  $\text{\rm End}(V)$ denote the $\F$-algebra consisting of the $\F$-linear
maps $V \to V$. 
We recall how each basis $\{v_i\}_{i=1}^n$ of $V$ gives an algebra isomorphism $\text{\rm End}(V) \to \Mat_n (\F)$.
For  $A \in {\rm End}(V)$ and $M \in \Mat_n(\F)$, we say that
{\em $M$ represents  $A$ with respect to $\{v_i\}_{i=1}^n$} whenever
 $A v_j = \sum_{i=1}^n M_{i,j} v_i$ for $1 \leq j \leq n$.
The isomorphism sends $A$ to the unique matrix in $\Mat_n(\F)$ that represents $A$ with
respect to $\{v_i\}_{i=1}^n$.
For a finite nonempty set $X$,
let $\Mat_X(\F)$ denote the $\F$-algebra consisting of the matrices
whose rows and columns are indexed by $X$ and with all entries in $\F$.
Let  $J \in \Mat_X(\F)$ have all entries $1$.
For $\A$, $\B \in \Mat_X(\F)$,
the {\em entry-wise product} $\A \circ \B\in \Mat_X(\F)$ is defined by
$(\A \circ \B)_{x,y} = \A_{x,y} \B_{x,y}$ for $x,y \in X$.
For $A \in \Mat_X(\F)$ and $B \in \Mat_Y(\F)$, the {\em Kronecker product}
$A \otimes B$ is the matrix in $\Mat_{X \times Y}(\F)$ such that 
\begin{align*}
(A \otimes B)_{ (x_1,y_1), (x_2, y_2)} 
 &= A_{x_1, x_2} B_{y_1,y_2} 
&&  (x_1, x_2 \in X,\; y_1, y_2 \in Y).
\end{align*}
For $A, A' \in\Mat_X(\F)$ and $B, B' \in \Mat_Y(\F)$,
\[
 (A \otimes B ) (A' \otimes B') = (A A') \otimes (B B').
\]

\section{Association schemes}
\label{sec:scheme}

For the rest of this paper, fix an integer $d \geq 1$.
By a {\em $d$-class symmetric association scheme} \cite[Section 2.1]{BCN} we mean a sequence
\[
 {\mathcal X} = (X, \{R_i\}_{i =0}^d),
\] 
where $X$ is a finite nonempty set and $R_0, R_1, \ldots, R_d$ are nonempty subsets of
the Cartesian product $X \times X$, that satisfy the following (i)--(iv).
\begin{itemize}
\item[\rm (i)]
$X \times X =  R_0 \cup R_1 \cup \cdots \cup R_d$ (disjoint union).
\item[\rm (ii)]
$R_0 = \{(x,x) \,|\, x \in X\}$.
\item[\rm (iii)]
$R_i^t = R_i$ for $0 \leq i \leq d$, where $R_i^t = \{(y,x) \,|\, (x,y) \in R_i\}$.
\item[\rm (iv)]
There exist integers $p^h_{i,j}$ $(0 \leq h,i,j \leq d)$
such that for $(x,y) \in R_h$,
\[
p^h_{i,j} = | \{ z \in X \,|\, (x,z) \in R_i, \; (z,y) \in R_j \} |.
\]
\end{itemize}
Note by (iii) that
\begin{itemize}
\item[\rm (v)]
$p^h_{i,j} = p^h_{j,i}  \qquad (0 \leq h,i,j \leq d)$.
\end{itemize}
The integers $p^h_{i,j}$  are called the {\em intersection numbers} of $\mathcal X$.

For the rest of this section, let $ {\mathcal X} = (X, \{R_i\}_{i =0}^d)$
denote a symmetric association scheme.
Define $k_i = p^0_{i,i}$ for $0 \leq i \leq d$.
By \cite[Lemma 2.1.1]{BCN},
\begin{align*}
p^h_{0,j} &= p^h_{j,0} = \delta_{j,h} \qquad(0 \leq h,j \leq d),
&
p^0_{i,j} &= \delta_{i,j} k_j \qquad (0\leq i,j \leq d).
\end{align*}
Let $\C$ denote the field of complex numbers.
For $0 \leq i \leq d$ define $\A_i \in \Mat_X (\C)$ such that
\begin{align*}
(\A_i)_{x,y} &=
 \begin{cases}
   1 & \text{\rm if $(x,y) \in R_i$},  
  \\
   0 & \text{\rm  if $(x,y) \not\in R_i$}
 \end{cases}
  &&  (x,y \in X).
\end{align*}
The matrix $A_i$ is called the {\em associate matrix} corresponding to $R_i$.
We have 
(i) $\sum_{i=0}^d \A_i = \J$;  
(ii)  $\A_0 = \I$; 
(iii) $\A_i = \A_i^t \;   (0 \leq i \leq d)$;               
(iv) $A_i \A_j = \sum_{h=0}^d p^h_{i,j} \A_h \; (0 \leq i,j \leq d)$;
(v) $\A_i \A_j = \A_j \A_i \; (0 \leq i,j \leq d)$. 
The matrices $\{\A_i\}_{i=0}^d$ form a basis of a commutative algebra $\mathcal M$,
called the {\em Bose-Mesner algebra} of $\mathcal X$ \cite[Section 2.2]{BCN}.
By \cite[p.\ 59, p.\ 64]{BI}, $\mathcal M$ has a second basis $\{\E_i\}_{i=0}^d$
such that
(i) $\sum_{i =0}^d \E_i = \I$; 
(ii) $\E_0 = |X|^{-1} \J$;
(iii) $\E_i^t = \E_i = \overline{\E}_i  \; (0 \leq i \leq d)$; 
(iv)$\E_i \E_j = \delta_{i,j} \E_i   \; (0 \leq i,j \leq d)$.  
We refer to $\{\E_i\}_{i=0}^d$ as the {\em  primitive idempotents} of $\mathcal X$.

We will be discussing the algebra $\Mat_{d+1}(\C)$.
For matrices in this algebra we index
the rows and columns by $0,1,\ldots, d$.
Define matrices $P$, $Q \in \Mat_{d+1} (\C)$ such that
\begin{align}
 \A_j &= \sum_{i=0}^d P_{i,j} \E_i     \qquad\qquad\qquad \;\; (0 \leq j \leq d),   \label{eq:Aj}
\\
 \E_j &= |X|^{-1} \sum_{i=0}^d Q_{i,j}  \A_i \qquad\qquad   (0 \leq j \leq d).    \label{eq:Ej}
\end{align}
Taking the complex conjugate in \eqref{eq:Aj} and \eqref{eq:Ej}, we observe
that the entries of $P$ and $Q$ are all real numbers.
By \cite[Section 2.2]{BCN}, $PQ = QP = |X| I$.
The matrix $P$ (resp.\ $Q$) is called the {\em first eigenmatrix} (resp.\ {\em second eigenmatrix}) 
{\em of $\mathcal X$ with respect to  the ordering $\{\E_i\}_{i=0}^d$}.
Since $\mathcal M$ is closed under the entry-wise product $\circ$,
there exist scalars $q^h_{i,j}\in\C$ $(0 \leq h,i,j \leq d)$
such that
\begin{align}
\E_i \circ \E_j &= |X|^{-1} \sum_{h=0}^d q^h_{i,j} \E_h   \qquad\qquad (0 \leq i,j \leq d).  \label{eq:EicircEj0}
\end{align}
The scalars $q^h_{i,j}$ are called the {\em Krein parameters} of $\mathcal X$
with respect to the ordering $\{\E_i\}_{i=0}^d$.
By \cite[Theorem 3.8]{BI} the Krein parameters are nonnegative real numbers.
Define $m_i = q^0_{i,i}$ for $0 \leq i  \leq d$.
By \cite[Lemma  2.3.1]{BCN},
\begin{align*}
 q^h_{0,j} &= q^h_{j,0} = \delta_{j,h}  \qquad  ( 0 \leq h,j \leq d),
&
q^0_{i,j} &= \delta_{i,j} m_j \qquad (0 \leq i,j \leq d).
\end{align*}
By \cite[Lemma 2.2.1]{BCN},
\begin{equation}
P_{0,j} = k_j, \qquad\qquad 
Q_{0,j} = m_j \qquad\qquad (0 \leq j \leq d).  \label{eq:Pi1}
\end{equation}
By \cite[Theorem 3.6]{BI},
\begin{align}
 p^h_{i,j} &= \frac{ k_i k_j }{ |X| } \sum_{r=0}^d \frac{1} { m_r^2} Q_{i,r} Q_{j,r} Q_{h,r}   
                  \qquad\qquad (0 \leq h,i,j \leq d),
                   \label{eq:phij0}
\\
 q^h_{i,j} &= \frac{ m_i m_j } {|X|}
                 \sum_{r=0}^d \frac{1}{k_r^2}  P_{i,r} P_{j,r} P_{h,r}     
                 \qquad\qquad  (0 \leq h,i,j \leq d).
                    \label{eq:qhij0}
\end{align}

\section{Formally self-dual and numerically self-dual association schemes}
\label{sec:dual}

Let $\mathcal X$ denote a symmetric $d$-class association scheme.
Fix an ordering $\{\E_i\}_{i=0}^d$ of the primitive idempotents of $\mathcal X$.
Let $q^h_{i,j}$ denote the Krein parameters of $\mathcal X$
with respect to the ordering $\{\E_i\}_{i=0}^d$.
Let $P$ (resp.\ $Q$) denote the first eigenmatrix (resp.\ second eigenmatrix) of $\mathcal X$
with respect to the ordering $\{\E_i\}_{i=0}^d$.
The association scheme  $\mathcal X$ is said to be {\em formally self-dual} (with respect to the ordering
$\{\E_i\}_{i=0}^d$) whenever $P = Q$ \cite[Section 2.3]{BCN}.
In this case we find by \eqref{eq:Pi1}--\eqref{eq:qhij0} that
\begin{align}
p^h_{i,j} &= q^h_{i,j}   \qquad\qquad  (0 \leq h,i,j \leq d).   \label{eq:phij=qhij}
\end{align}
This raises a natural question.
For a given ordering $\{\E_i\}_{i=0}^d$ of the primitive idempotents,
is it  possible that \eqref{eq:phij=qhij} holds but $P \neq Q$?
This is possible as we will show.
In the next section we will demonstrate with an example of an association scheme and an ordering
of the primitive idempotents with respect to which the intersection numbers and the Krein parameters
are the same but the first eigenmatrix and the second eigenmatrix are different.
We will use the following term.

\begin{defi}   \label{def:num}   \samepage
\ifDRAFT {\rm def:num}. \fi
The association scheme  $\mathcal X$ is said to be {\em numerically  self-dual} 
(with respect to the given ordering $\{\E_i\}_{i=0}^d$ of the primitive idempotents) 
whenever $p^h_{i,j} = q^h_{i,j}$  for $0 \leq h,i,j \leq d$.  
\end{defi}

Motivated  by our above  comments, we now consider what happens to the $P$, $Q$ and
the Krein parameters  if we change the ordering of the primitive idempotents.
Let $\sigma$ denote a permutation of $\{0,1,2, \ldots, d\}$ that fixes $0$.
Then $\{\E_{\sigma(i)} \}_{i=0}^d$ is an ordering of the primitive idempotents of $\mathcal X$.
Below we describe the first and second eigenmatrices and the Krein parameters with respect to this ordering.
These matrices and parameters will be expressed in bold font.

\begin{lemma}    \label{lem:Ps}   \samepage
\ifDRAFT {\rm lem:Ps}. \fi
Let $\sigma$ denote a permutation of $\{0,1,2, \ldots, d\}$ that fixes $0$.
Let $\P$ (resp.\ $\Q$) denote the first eigenmatrix (resp.\ second eigenmatrix) of $\mathcal X$
with respect to the ordering $\{\E_{\sigma(i)} \}_{i=0}^d$.
Then
\begin{align}
  \P_{i,j} &= P_{\sigma(i), j},   &
  \Q_{i,j} &= Q_{i, \sigma(j)}    &&   (0 \leq i,j \leq d).
       \label{eq:Pdij}
\end{align}
Let $T \in \Mat_{d+1}(\C)$ denote the permutation matrix coresponding to $\sigma$:
$T_{i,j} = \delta_{\sigma(i), \, j}$ $(0 \leq i,j \leq d)$. 
Then
\begin{align}
  \P& = T P, &
  \Q = Q T^{-1}.   \label{eq:TP}
\end{align}
\end{lemma}

\begin{proof}
In \eqref{eq:Aj},  replace $i$ with $\sigma(i)$ to get the equation  on the left in \eqref{eq:Pdij}.
In \eqref{eq:Ej}, replace $j$ with $\sigma(j)$ to get 
the equation on the right in \eqref{eq:Pdij}.
\end{proof}

\begin{lemma}   \label{lem:qd}   \samepage
\ifDRAFT {\rm lem:qd}. \fi
Let $\sigma$ denote a permutation of $\{0,1,2, \ldots, d\}$ that fixes $0$.
Let $\q^h_{i,j}$ denote the Krein parameters of $\mathcal X$
with respect to the ordering $\{\E_{\sigma(i)} \}_{i=0}^d$.
Then 
\begin{align}
  \q^h_{i,j} &= q^{\sigma(h)}_{\sigma(i),\; \sigma(j)}   \qquad\qquad  ( 0 \leq  h,i,j \leq d).
                        \label{eq:q}
\end{align}
\end{lemma}

\begin{proof}
In  \eqref{eq:EicircEj0},
replace $h$, $i$, $j$ with $\sigma(h)$, $\sigma(i)$, $\sigma(j)$, respectively.
This gives \eqref{eq:q}.
\end{proof}

\section{An example of a scheme}
\label{sec:ex}

There is a type of association scheme called a group scheme \cite{BM}.
In this section, we discuss the group scheme attached to a certain abelian group 
that we now define.
Consider the additive group $\Z_2 = \Z/2 \Z = \{0,1\}$.
Fix an integer $m \geq 1$. We define an abelian group $G = G^{(m)}$ by
\begin{align}
 G &= \Z_2 \oplus \cdots \oplus \Z_2  \qquad \text{\rm ($m$ copies)}.     \label{eq:defG}
\end{align}
Note that $|G| = 2^m$.
We view each element  of $G$ as a sequence $a_1 a_2 \cdots a_m$, where
$a_i \in \Z_2$ for $1 \leq i \leq m$.
We now describe the group scheme attached to $G$.
In this description, it is convenient to put the elements of $G$ in lexicographical order.
For example, if $m=2$ then the order is $00 < 01 < 10 < 11$.
For $x \in G$ define a relation $R_x \subseteq G \times G$ by
\begin{align*}
R_x = \{(y,y+x)\,|\, y \in G \}.
\end{align*}
Then $(G, \{R_x\}_{x \in G})$ forms a $d$-class symmetric association scheme 
with $d=|G|-1$,
called the group scheme for $G$; see \cite{BM}.
We denote this association scheme by ${\mathcal X}^{(m)}$.
Below we describe the structure of ${\mathcal X}^{(m)}$.

We first consider the case $m=1$.
The associate matrices of ${\mathcal X}^{(1)}$ are
\begin{align*}
  \A^{(1)}_0 &= \begin{pmatrix} 1 & 0 \\ 0 & 1 \end{pmatrix},  &
  \A^{(1)}_1&= \begin{pmatrix} 0 & 1 \\ 1 & 0  \end{pmatrix}.
\end{align*}
The primitive idempotents of ${\mathcal X}^{(1)}$ are
\begin{align*}
\E^{(1)}_0 &= \frac{1}{2} \begin{pmatrix} 1 & 1 \\ 1 & 1 \end{pmatrix},  &
 \E^{(1)}_1 &= \frac{1}{2} \begin{pmatrix} 1 & -1 \\ -1 & 1 \end{pmatrix}.
\end{align*}
The first and second eigenmatrices of ${\mathcal X}^{(1)}$ are
\begin{align*}
 P^{(1)} &= \begin{pmatrix} 1 & 1 \\ 1 & -1 \end{pmatrix},  &
 Q^{(1)} &= \begin{pmatrix} 1 & 1 \\ 1 & -1 \end{pmatrix}.
\end{align*}
Note that $P^{(1)}=Q^{(1)}$.
On occasion, it is convenient to view $\Z_2$ as the field with two elements.
Using this point of view,
\begin{align}
(P^{(1)})_{x,y} &= (-1)^{xy},   & (Q^{(1)})_{x,y} &= (-1)^{xy}  &&  (x,y \in \Z_2).   \label{eq:P1xy}
\end{align}

For the rest of this section, assume that $m \geq 2$.
The associate matrices and the primitive idempotents of ${\mathcal X}^{(m)}$
are described as follows.
Pick $x \in G$ and write $x=x_1 x_2 \cdots x_m$.
Then
\begin{align}
\A^{(m)}_x & = \A^{(1)}_{x_1} \otimes  \A^{(1)}_{x_2} \otimes \cdots \otimes   \A^{(1)}_{x_m},
                       \label{eq:defAx}
\\
\E^{(m)}_x &= \E^{(1)}_{x_1} \otimes  \E^{(1)}_{x_2} \otimes \cdots \otimes   \E^{(1)}_{x_m}.   \label{eq:defEx}
\end{align}
For $x,y \in G$ we have
\begin{align*}
  \A^{(m)}_x \A^{(m)}_y &= \A^{(m)}_{x+y},   &
  \E^{(m)}_x \circ \E^{(m)}_y &= |G|^{-1} \E^{(m)}_{x+y}.
\end{align*}
The intersection numbers and Krein  parameters of  ${\mathcal X}^{(m)}$
are described as follows.
For $x,y,z \in G$ we have 
\begin{align}
   p^z_{x,y} &= \delta_{x+y,\,z},    & q^z_{x,y} &= \delta_{x+y,\,z}.   
 \label{eq:pcab} 
\end{align}
The first and the second eigenmatrices of ${\mathcal X}^{(m)}$ are
described as follows.
\begin{align*}
P^{(m)} &= P^{(1)} \otimes P^{(1)}\otimes \cdots \otimes P^{(1)} \qquad\qquad (\text{$m$ copies}),
\\
Q^{(m)} &= Q^{(1)} \otimes Q^{(1)}\otimes \cdots \otimes Q^{(1)}  \qquad\qquad (\text{$m$ copies}).
\end{align*}
Note that $P^{(m)} = Q^{(m)}$.

Next we obtain a formula for the entries of $P^{(m)}$ and $Q^{(m)}$
that is analogous to \eqref{eq:P1xy}.
Pick $x,y \in G$ and write
\begin{align*}
x &= x_1 x_2 \cdots x_m,    & y =& y_1 y_2 \cdots y_m.
\end{align*}
Then
\begin{align}
  (P^{(m)})_{x,y}   &= \be{x,y},  
&
 (Q^{(m)})_{x,y} &= \be{x,y},            \label{eq:Qxy}
\end{align}
where
\[
   \bilin{x,y} = \sum_{i=1}^m x_i y_i.
\]
The group $G$ becomes a vector space over the field $\Z_2$ in the following way.
For a scalar 
$\alpha$ in $\Z_2$ and an element $x = x_1 x_2 \cdots x_m$ in $G$, 
$\alpha x = x'_1 x'_2 \cdots x'_m$
where $x'_i = \alpha x_i$ for $1 \leq i \leq m$.
Note that the map $\bilin{ \, , \,} : G \times G \to \Z_2$ is a symmetric bilinear form on the vector space $G$.

Our next goal is to determine the orderings of the primitive idempotents of ${\mathcal X}^{(m)}$
with respect to which
${\mathcal X}^{(m)}$ is numerically self-dual.
 
\begin{lemma}   \label{lem:linear}   \samepage
\ifDRAFT {\rm lem:linear}. \fi
For a map $\sigma : G \to G$
the following are equivalent:
\begin{itemize}
\item[\rm (i)]
$\sigma$ is $\Z_2$-linear;
\item[\rm (ii)]
$\sigma(x+y) = \sigma(x) + \sigma(y)$ for  $x,y \in G$.
\end{itemize}
Suppose {\rm (i)} and {\rm (ii)} hold.
Then $\sigma$ fixes $0$.
\end{lemma}

\begin{proof}
(i) $\Rightarrow$ (ii)
Clear.

(ii) $\Rightarrow$ (i)
By setting $x=y=0$ in (ii), we get $\sigma(0)=0$.
We show that $\sigma(\alpha x) = \alpha \sigma(x)$ for $\alpha \in \Z_2$ and $x \in G$.
If $\alpha = 1$ then the both sides are $\sigma (x)$. If $\alpha=0$ then both sides are $0$ since
$\sigma(0)=0$.
\end{proof}

\begin{theorem}    \label{thm:main1}    \samepage
\ifDRAFT {\rm thm:main1}. \fi
For a bijection $\sigma : G \to G$ 
the following are equivalent:
\begin{itemize}
\item[\rm (i)]
$\sigma(0)=0$ and ${\mathcal X}^{(m)}$ is numerically self-dual with respect to the ordering 
$\E^{(m)}_{\sigma(x)}$ $(x \in G)$;
\item[\rm (ii)]
$\sigma$ is $\Z_2$-linear.
\end{itemize}
\end{theorem}

\begin{proof}
Using the equation on the left in \eqref{eq:pcab}, we obtain
\begin{align}
 p^z_{x,y} &= \delta_{x+y, z} =\delta_{\sigma(x+y), \, \sigma(z)}   \qquad\qquad    (x,y,z \in G).   \label{eq:mainaux1}
\end{align}
Let $\q^z_{x,y}$  denote the Krein parameters of ${\mathcal X}^{(m)}$
with respect to the ordering $\E^{(m)}_{\sigma(x)}$ $(x \in G)$.
Using Lemma \ref{lem:qd} and the equation on the right in \eqref{eq:pcab}, we obtain
 \begin{align}
 \q^z_{x,y}  &= q^{\sigma(z)}_{\sigma(x), \,\sigma(y)} = \delta_{\sigma(x)+\sigma(y),\, \sigma(z)}
            \qquad\qquad   (x,y,z \in G).          \label{eq:mainaux2}
\end{align}
Let $x,y \in G$.
Comparing \eqref{eq:mainaux1} and \eqref{eq:mainaux2}, we find that
$\sigma(x+y) = \sigma(x) + \sigma(y)$ if and only if
$p^z_{x,y} = \q^z_{x,y}$ for $z \in G$.
By this and Lemma \ref{lem:linear} we get the result.
\end{proof}

Our next goal is to determine the orderings of the primitive idempotents of ${\mathcal X}^{(m)}$
with respect to which
${\mathcal X}^{(m)}$ is formally self-dual.
Let $\sigma : G \to G$ denote a map.
Shortly we will give a result involving the condition
\begin{align}
  \bilin{\sigma(x),y } &= \bilin{x, \sigma(y)}  \qquad\qquad (x,y \in G).    \label{eq:condsig}
\end{align}
To prepare for this result, we have some comments about \eqref{eq:condsig}.

Referring to the group $G  = \Z_2 \oplus \cdots \oplus \Z_2$,
for $1 \leq i \leq m$ let $e_i$ denote the generator of the $i$th copy of $\Z_2$.
Thus
\[
e_i = 0\cdots 01 0 \cdots 0,
\]
where the $1$ is in coordinate $i$.
Note that $\{e_i\}_{i=1}^m$ is an orthonormal basis of the vector space $G$.
Also note that 
\begin{align}
  x&= \sum_{i=1}^m \bilin{x, e_i} e_i  \qquad\qquad  (x \in G).   
 \label{eq:xei}
\end{align}

\begin{lemma}   \label{lem:linear2}   \samepage
\ifDRAFT {\rm lem:linear2}. \fi
Let $\sigma : G \to G$ denote a map that satisfies the condition \eqref{eq:condsig}.
Then $\sigma$ is $\Z_2$-linear.
\end{lemma}

\begin{proof}
We invoke Lemma \ref{lem:linear}.
For $x,y \in G$ we show that
\begin{align}
    \sigma(x+y) = \sigma(x) +\sigma(y).    \label{eq:linear2aux}
\end{align}
Using \eqref{eq:condsig} and \eqref{eq:xei} we obtain
\begin{align*}
 \sigma(x)  &= \sum_{i=1}^m \bilin{x, \sigma(e_i) } e_i.
\end{align*}
Similarly
\begin{align*}
 \sigma(y)  &= \sum_{i=1}^m \bilin{y, \sigma(e_i) } e_i,
\\
 \sigma(x+y)  &= \sum_{i=1}^m \bilin{x+y, \sigma(e_i) } e_i.
\end{align*}
For $1 \leq i \leq m$,
\[
    \bilin{x, \sigma(e_i)}  +  \bilin{y, \sigma(e_i)}
  = \bilin{x+y, \sigma(e_i)}.
\]
By these comments we obtain \eqref{eq:linear2aux}.
\end{proof}

Let $\sigma : G \to G$ denote a $\Z_2$-linear map.
Let the matrix $S \in \Mat_m(\Z_2)$ represent $\sigma$
with respect to the basis $\{e_i\}_{i=1}^m$.
Thus
\begin{align*}
  \sigma(e_j) &= \sum_{i=1}^m S_{i,j} e_i  \qquad\qquad  (1 \leq j \leq m).
\end{align*}
Using this and \eqref{eq:xei}, we find that for $1 \leq i,j \leq m$,
\begin{equation}
  S_{i,j} =  \bilin{ e_i, \sigma(e_j)} = \bilin{\sigma(e_j), e_i }.        \label{eq:esamplepre}
\end{equation}

\begin{lemma}   \label{lem:example}    \samepage
\ifDRAFT {\rm lem:example}. \fi
For a map $\sigma : G \to G$ the following are equivalent:
\begin{itemize}
\item[\rm (i)]
$\sigma$ satisfies \eqref{eq:condsig};
\item[\rm (ii)]
$\sigma$ is $\Z_2$-linear,
and the matrix representing $\sigma$ with respect to $\{e_i\}_{i=1}^m$ is symmetric. 
\end{itemize}
\end{lemma}

\begin{proof}
By Lemma \ref{lem:linear} we may assume that $\sigma$ is $\Z_2$-linear.
Let the matrix $S \in \Mat_m(\Z_2)$ represent  $\sigma$ with respect to $\{e_i\}_{i=1}^m$.
By \eqref{eq:esamplepre}, the condition \eqref{eq:condsig} holds if and only if
$S_{i,j} = S_{j,i}$ for $1 \leq i,j \leq m$.
\end{proof}

\begin{theorem}    \label{thm:main3}    \samepage
\ifDRAFT {\rm thm:main3}. \fi
For a bijection $\sigma : G \to G$  the following are equivalent:
\begin{itemize}
\item[\rm (i)]
$\sigma(0) = 0$ and  ${\mathcal X}^{(m)}$ is formally self-dual with respect to the ordering 
$\E^{(m)}_{\sigma(x)}$ $(x \in G)$;
\item[\rm (ii)]
$\sigma$ satisfies \eqref{eq:condsig}.
\end{itemize}
\end{theorem}

\begin{proof}
By Lemma \ref{lem:linear2}
we may assume that $\sigma(0) =0$.
Let $\P^{(m)}$ (resp.\ $\Q^{(m)}$) denote the first eigenmatrix
(resp.\ second eigenmatrix) of ${\mathcal X}^{(m)}$ with respect to the ordering
$\E^{(m)}_{\sigma(x)}$ $(x \in G)$. 
Pick $x,y \in G$.
By Lemma \ref{lem:Ps} and \eqref{eq:Qxy},
\begin{align*}
 (\P^{(m)})_{x,y} &= \be{\sigma(x), y}, 
&
(\Q^{(m)})_{x,y} &= \be{x, \sigma(y)} .
\end{align*}
Thus $(\P^{(m)})_{x,y}  = (\Q^{(m)})_{x,y}$ if and only if 
$\be{\sigma(x),y } = \be{x, \sigma(y)}$
if and only if
$\bilin{\sigma(x),y } = \bilin{x, \sigma(y)}$.
The result is a routine consequence of this.
\end{proof}

\begin{corollary}     \label{cor:main}    \samepage
\ifDRAFT {\rm cor:main}. \fi
Let $\sigma : G\to G$ denote a $\Z_2$-linear bijection.
Then  $\E^{(m)}_{\sigma(x)}$  $(x \in G)$ is an ordering of the primitive idempotents
of ${\mathcal X}^{(m)}$ with respect to which  ${\mathcal X}^{(m)}$ is numerically self-dual.
Let the matrix $S \in \Mat_m(\Z_2)$ represent  $\sigma$ with respect to $\{e_i\}_{i=1}^m$
Then ${\mathcal X}^{(m)}$ is formally self-dual 
with respect to the ordering $\E^{(m)}_{\sigma(x)}$  $(x \in G)$
if and only if $S$ is symmetric.
\end{corollary}

\begin{proof}
With respect to the ordering $E^{(m)}_{\sigma(x)}$ $(x \in G)$, ${\mathcal X}^{(m)}$ is
numerically self-dual by Theorem \ref{thm:main1}.
By Lemma  \ref{lem:example} and Theorem \ref{thm:main3}, 
${\mathcal X}^{(m)}$ is formally self-dual if and only if $S$ is symmetric.
\end{proof}

In the following example, we examine the case  $m=2$ in more detail.
We display all the orderings of the primitive idempotents with respect to which
${\mathcal X}^{(2)}$ is numerically self-dual;
there are six such orderings.
We show that with respect to four of these orderings, ${\mathcal X}^{(2)}$ is formally-self dual,
and with respect to the remaining two orderings, ${\mathcal X}^{(2)}$ is not formally self-dual.

\begin{example}    \label{example} 
\ifDRAFT {\rm example}. \fi
Referring to \eqref{eq:defG},
we consider the case of $m=2$, so
\[
   G = \Z_2 \oplus \Z_2.
\]
We put the elements of $G$ in the lexicographic order: $00<01<10< 11$.
By \eqref{eq:defAx},
the associate matrices are
\begin{align*}
\A^{(2)}_{00} &=
\A^{(1)}_0 \otimes \A^{(1)}_0
= \begin{pmatrix}
    1 & 0 \\
    0 & 1
  \end{pmatrix}
\otimes
 \begin{pmatrix}
    1 & 0  \\
    0 & 1
  \end{pmatrix}
= \begin{pmatrix}
       1 & 0 & 0 & 0 \\
      0 & 1 & 0 & 0 \\
       0 & 0 & 1 & 0 \\
     0 & 0 & 0 & 1
   \end{pmatrix},
\\
\A^{(2)}_{01} &= 
\A^{(1)}_0 \otimes \A^{(1)}_1
= \begin{pmatrix}
    1 & 0 \\
    0 & 1
  \end{pmatrix}
\otimes
 \begin{pmatrix}
    0 & 1  \\
    1 &0
  \end{pmatrix}
=
\begin{pmatrix}
            0 & 1 & 0 & 0 \\
            1 & 0 & 0 & 0 \\
            0 & 0 & 0 & 1 \\
            0 & 0 & 1 & 0
     \end{pmatrix},\\
\A^{(2)}_{10}&= 
\A^{(1)}_1 \otimes \A^{(1)}_0
= \begin{pmatrix}
    0 & 1 \\
    1 & 0
  \end{pmatrix}
\otimes
 \begin{pmatrix}
    1 & 0  \\
    0 & 1
  \end{pmatrix}
=
\begin{pmatrix}
         0 & 0 & 1 & 0  \\ 
        0 & 0 & 0 & 1 \\
         1 & 0 & 0 & 0  \\
        0 & 1 & 0 & 0
       \end{pmatrix},
\\
\A^{(2)}_{ 11}&=
\A^{(1)}_1 \otimes \A^{(1)}_1
= \begin{pmatrix}
    0 & 1 \\
    1 & 0
  \end{pmatrix}
\otimes
 \begin{pmatrix}
    0 & 1  \\
    1 & 0
  \end{pmatrix}
=
 \begin{pmatrix}
      0 & 0 & 0 & 1 \\
      0 & 0 &  1 & 0 \\
     0 & 1 & 0 & 0 \\
    1 & 0 & 0 & 0
\end{pmatrix}.
\end{align*}
By \eqref{eq:defEx}, the primitive idempotents are
\begin{align*}
\E^{(2)}_{00} &= 
\E^{(1)}_0 \otimes \E^{(1)}_0
= \frac{1}{2}
   \begin{pmatrix}
    1 & 1 \\
    1 & 1
\end{pmatrix}
 \otimes
\frac{1}{2}
 \begin{pmatrix}
   1 & 1 \\
    1 & 1
  \end{pmatrix}
=
\frac{1}{4}
 \begin{pmatrix}
   1 & 1 & 1 & 1 \\
   1 & 1 & 1 & 1 \\
   1 & 1 & 1 & 1  \\
  1 & 1 & 1 & 1 
\end{pmatrix},
\\
 \E^{(2)}_{01} &= 
\E^{(1)}_0 \otimes \E^{(1)}_1
= \frac{1}{2}
   \begin{pmatrix}
    1 & 1 \\
    1 & 1
\end{pmatrix}
 \otimes
\frac{1}{2}
 \begin{pmatrix}
   1 & -1 \\
    -1 & 1
  \end{pmatrix}
=
 \frac{1}{4}
      \begin{pmatrix}
     1 & -1 & 1 & -1  \\
     -1 & 1 & -1 & 1  \\
    1 & -1 & 1 & -1  \\
   -1 & 1 & -1 & 1
  \end{pmatrix},
\\
E^{(2)}_{10} &=
\E^{(1)}_1 \otimes \E^{(1)}_0
= \frac{1}{2}
   \begin{pmatrix}
    1 & -1 \\
    -1 & 1
\end{pmatrix}
 \otimes
\frac{1}{2}
 \begin{pmatrix}
   1 & 1 \\
    1 & 1
  \end{pmatrix}
=
\frac{1}{4}
 \begin{pmatrix}
    1 & 1 & -1 & -1 \\
    1 & 1 & -1 & -1  \\
   -1 & -1 & 1 & 1 \\
   -1 & -1 & 1 & 1
\end{pmatrix},
\\
\E^{(2)}_{11} &=
\E^{(1)}_1 \otimes \E^{(1)}_1
= \frac{1}{2}
   \begin{pmatrix}
    1 & -1 \\
    -1 & 1
\end{pmatrix}
 \otimes
\frac{1}{2}
 \begin{pmatrix}
   1 & -1 \\
   - 1 & 1
  \end{pmatrix}
=\frac{1}{4}
 \begin{pmatrix}
   1 & -1 & -1 & 1  \\
   -1 & 1 & 1 & -1  \\
   -1 & 1 & 1 & -1 \\
   1 & -1 & -1 & 1
\end{pmatrix}.
\end{align*}
By \eqref{eq:Qxy}, 
the first eigenmatrix $P^{(2)}$ and the second eigenmatrix $Q^{(2)}$ are equal to
\[
   \begin{pmatrix}
    1 & 1 & 1 & 1 \\
    1 & -1 & 1 & -1 \\
   1 & 1 & -1 & -1 \\
    1 & -1 & -1 & 1
   \end{pmatrix}.
\]
We will display all the orderings of the primitive idempotents with respect to which
${\mathcal X}^{(2)}$ is numerically self-dual.
There are six such orderings.
With respect to four of these orderings,  ${\mathcal X}^{(2)}$ is formally self-dual,
and with respect to the remaining two orderings,  ${\mathcal X}^{(2)}$ is not formally self-dual.
As in Section \ref{sec:ex},
let us view   $G$ as a vector space over the field $\Z_2$.
For a $\Z_2$-linear bijection $\sigma : G \to G$, consider the ordering $\E^{(2)}_{\sigma(x)}$ $(x \in G)$.
With respect to this ordering, ${\mathcal X}^{(2)}$ is numerically self-dual.
We now consider if ${\mathcal X}^{(2)}$ is formally self-dual with respect to this ordering.
Let $S_\sigma \in \Mat_m(\Z_2)$ represent $\sigma$ with respect to the basis $e_1,e_2$.
There are six solutions for $S_\sigma$; four symmetric and two non-symmetric.
Below we list the solutions.
For each solution we give $\sigma(x)$ $(x \in G)$ along with the first and second eigenmatrix
with respect to $\E^{(2)}_{\sigma(x)}$ $(x \in G)$.
We now give the symmetric solutions for  $S_\sigma$.
{\small
\[
\begin{array}{c|cccccc}
S_\sigma&  \sigma(00), \sigma(01), \sigma(10), \sigma(11) & \text{$1$st emat wrt \ $\{\E^{(2)}_{\sigma(x)}\}_{x \in G}$}
  & \text{$2$nd emat wrt  $\{\E^{(2)}_{\sigma(x)}\}_{x \in G}$}
\\ \hline
   \begin{pmatrix}
     1 & 0 \\ 
     0 & 1
  \end{pmatrix}
& 00, 01, 10, 11 &
\begin{pmatrix}
    1 & 1 & 1 & 1 \\
    1& -1 & 1 & -1 \\
    1 & 1 & -1 & -1 \\
    1 & -1 & -1 & 1
\end{pmatrix}
&
\begin{pmatrix}
    1 & 1 & 1 & 1 \\
    1& -1 & 1 & -1 \\
    1 & 1 & -1 & -1 \\
    1 & -1 & -1 & 1
\end{pmatrix}                  \rule{0mm}{8.5ex}
\\
   \begin{pmatrix}
     0 & 1 \\ 
     1 & 0
  \end{pmatrix}
& 00, 10, 01, 11 &
\begin{pmatrix}
    1 & 1 & 1 & 1 \\
    1& 1 & -1 & -1 \\
    1 & -1 & 1 & -1 \\
    1 & -1 & -1 & 1
\end{pmatrix}
&
\begin{pmatrix}
    1 & 1 & 1 & 1 \\
    1& 1 & -1 & -1 \\
    1 & -1 & 1 & -1 \\
    1 & -1 & -1 & 1
\end{pmatrix}          \rule{0mm}{8.5ex}
\\
   \begin{pmatrix}
     0 & 1 \\ 
     1 & 1
  \end{pmatrix}
& 00, 11, 01, 10 &
\begin{pmatrix}
    1 & 1 & 1 & 1 \\
    1& -1 & -1 & 1 \\
    1 & -1 & 1 & -1 \\
    1 & 1 & -1 & -1
\end{pmatrix}
&
\begin{pmatrix}
    1 & 1 & 1 & 1 \\
    1& -1 & 1 & -1 \\
    1 & -1 & 1 & -1 \\
    1 & 1 & -1 & -1
\end{pmatrix}          \rule{0mm}{8.5ex}
\\
   \begin{pmatrix}
     1 & 1 \\ 
     1 & 0
  \end{pmatrix}
& 00, 10, 11, 01 &
\begin{pmatrix}
    1 & 1 & 1 & 1 \\
    1& 1 & -1 & -1 \\
    1 & -1 & -1 & 1 \\
    1 & -1 & 1 & -1
\end{pmatrix}
&
\begin{pmatrix}
    1 & 1 & 1 & 1 \\
    1& 1 & -1 & -1 \\
    1 & -1 & -1 & 1 \\
    1 & -1 & 1 & -1
\end{pmatrix}          \rule{0mm}{8.5ex}
\end{array}
\]
}
We now give the non-symmetric solutions for  $S_\sigma$.
{\small
\[
\begin{array}{c|cccccc}
S_\sigma& \sigma(00), \sigma(01), \sigma(10), \sigma(11) & \text{$1$st emat wrt $\{\E^{(2)}_{\sigma(x)}\}_{x \in G}$}
  & \text{$2$nd emat wrt $\{\E^{(2)}_{\sigma(x)}\}_{x \in G}$}
\\ \hline
   \begin{pmatrix}
     1 & 0 \\ 
     1 & 1
  \end{pmatrix}
& 00, 01, 11, 10 &
\begin{pmatrix}
    1 & 1 & 1 & 1 \\
    1& -1 & 1 & -1 \\
    1 & -1 & -1 & 1 \\
    1 & 1 & -1 & -1
\end{pmatrix}
&
\begin{pmatrix}
    1 & 1 & 1 & 1 \\
    1& -1 & -1 & 1 \\
    1 & 1 & -1 & -1 \\
    1 & -1 & 1 & -1
\end{pmatrix}          \rule{0mm}{8.5ex}
\\
   \begin{pmatrix}
     1 & 1 \\ 
     0 & 1
  \end{pmatrix}
& 00, 11, 10, 01 &
\begin{pmatrix}
    1 & 1 & 1 & 1 \\
    1& -1 & -1 & 1 \\
    1 & 1 & -1 & -1 \\
    1 & -1 & 1 & -1
\end{pmatrix}
&
\begin{pmatrix}
    1 & 1 & 1 & 1 \\
    1& -1 & 1 & -1 \\
    1 & -1 & -1 & 1 \\
    1 & 1 & -1 & -1
\end{pmatrix}          \rule{0mm}{8.5ex}
\end{array}
\]
}
The above tables show that $S_\sigma$ is symmetric if and only if ${\mathcal X}^{(2)}$ is formally
self-dual with respect to the ordering $\E^{(2)}_{\sigma(x)}$ $(x \in G)$.
\end{example}

\section{$P$-polynomial and $Q$-polynomial association schemes}
\label{sec:Ppoly}

There are certain properties of an association scheme, called  $P$-polynomial and $Q$-polynomial.
In this section we consider the formal self-duality and the numerical self-duality for
$P$-polynomial schemes and   $Q$-polynomial schemes.
Let ${\mathcal X} = (X, \{R_i\}_{i=0}^d)$ denote a symmetric association scheme.

\begin{defi}    \label{def:Ppoly}    \samepage
\ifDRAFT {\rm def:Ppoly}. \fi
The scheme $\mathcal X$ is said to be {\em $P$-polynomial}
whenever the intersection numbers $p^h_{i,j}$ satisfy the 
the following conditions for $0 \leq h,i,j \leq d$:
\begin{itemize}
\item[\rm (i)]
$p^h_{i,j} = 0$ if  one of $h$, $i$, $j$ is greater than  the sum of the
other two;
\item[\rm (ii)]
$p^h_{i,j} \neq 0$ if  one of $h$, $i$, $j$ is equal to  the sum of the
other two.
\end{itemize}
\end{defi}

\begin{theorem}    \label{thm:main2}    \samepage
\ifDRAFT {\rm thm:main2}. \fi
Assume that $\mathcal X$ is $P$-polynomial.
Fix an ordering  $\{E_i\}_{i=0}^d$ of the primitive idempotents of $\mathcal X$.
Then  the following are equivalent:
\begin{itemize}
\item[\rm (i)]
$\mathcal X$ is formally self-dual with respect to the ordering $\{E_i\}_{i=0}^d$;
\item[\rm (ii)]
$\mathcal X$ is numerically self-dual with respect to the ordering $\{E_i\}_{i=0}^d$.
\end{itemize}
\end{theorem}

The proof of Theorem \ref{thm:main2} will be completed in Section \ref{sec:proof}.

\begin{defi}    \label{def:Qpoly}    \samepage
\ifDRAFT {\rm def:Qpoly}. \fi
An ordering $\{E_i\}_{i=0}^d$ of the primitive idempotents is said to be  {\em $Q$-polynomial}
whenever the corresponding Krein parameters $q^h_{i,j}$
satisfy the following conditions for $0 \leq h,i,j \leq d$:
\begin{itemize}
\item[\rm (i)]
$q^h_{i,j} = 0$ if  one of $h$, $i$, $j$ is greater than  the sum of the
other two;
\item[\rm (ii)]
$q^h_{i,j} \neq 0$ if  one of $h$, $i$, $j$ is equal to  the sum of the
other two.
\end{itemize}
\end{defi}

\begin{theorem}    \label{thm:main4}    \samepage
\ifDRAFT {\rm thm:main4}. \fi
Let $\{E_i\}_{i=0}^d$ denote a  $Q$-polynomial ordering of the primitive idempotents of 
$\mathcal X$.
Then the following are equivalent:
\begin{itemize}
\item[\rm (i)]
$\mathcal X$ is formally self-dual  with respect to the  ordering $\{E_i\}_{i=0}^d$;
\item[\rm (ii)]
$\mathcal X$ is numerically self-dual with respect to the  ordering $\{E_i\}_{i=0}^d$.
\end{itemize}
\end{theorem}

The proof of Theorem \ref{thm:main4} will be completed in Section \ref{sec:proof}.
In the proof of the above theorems  we use
the concept of  Askey-Wilson  duality.
In the next section we recall this concept.

\section{Askey-Wilson duality}
\label{sec:duality}

In this section, we recall the concept of Askey-Wilson duality.
Let ${\mathcal X} = (X, \{R_i\}_{i=0}^d)$ denote a symmetric association scheme.
Fix an ordering $\{E_i\}_{i=0}^d$ of the primitive idempotents of $\mathcal X$.
Recall the scalars $k_i = p^0_{i,i}$ and $m_i = q^0_{i,i}$.

Assume that $\mathcal X$ is $P$-polynomial.
We abbreviate
\begin{align*}
c_i &= p^i_{1,i-1} \;\; (1 \leq i \leq d),  &
a_i &= p^i_{1,i } \;\; (0 \leq i \leq d),   &
b_i &= p^i_{1,i+1} \;\; (0 \leq  i \leq d-1).
\end{align*}
For an indeterminate $\lambda$, let $\R[\lambda]$ denote the $\R$-algebra
consisting of the polynomials in $\lambda$ that have all coefficients in $\R$.
Define polynomials $u_i \in \R[\lambda]$ $(0 \leq i \leq d)$ by
$u_0 = 1$, $u_1 =\lambda/k_1$, and
\begin{align*}
  \lambda u_i &= c_i u_{i-1} + a_i u_i + b_i u_{i+1}      
              \qquad\qquad   ( 1 \leq  i \leq d-1).
\end{align*}
Abbreviate
\begin{align*}
  \th_i &= P_{i,1}  \qquad\qquad  (0 \leq i \leq d).
\end{align*}

\begin{lemma}   {\rm (See \cite[p.\ 190]{BI}.)}
\label{lem:Pij}   \samepage
\ifDRAFT {\rm lem:Pij}. \fi
We have
\begin{align*}
   P_{i,j} &= k_j u_j (\th_i)    \qquad\qquad (0 \leq i,j \leq d). 
\end{align*}
\end{lemma}

Let  $\{E_i\}_{i=0}^d$ denote a $Q$-polynomial ordering of the primitive idempotents of 
$\mathcal X$.
We abbreviate
\begin{align*}
c^*_i &= q^i_{1,i-1} \;\; (1 \leq i \leq d),  &
a^*_i &= q^i_{1,i} \;\; (0 \leq i \leq d),   &
b^*_i &= q^i_{1,i+1} \;\; (0 \leq  i \leq d-1).
\end{align*}
Define polynomials $u^*_i \in \R[\lambda]$ $(0 \leq i \leq d)$ by
$u^*_0 = 1$, $u^*_1 =\lambda/m_1$, and
\begin{align*} 
\lambda u^*_i &= c^*_{i} u^*_{i-1} + a^*_i u^*_i
        + b^*_i  u^*_{i+1}               \qquad\qquad   ( 1 \leq  i \leq d-1).
\end{align*}
Abbreviate
\begin{align*}
  \th^*_i &=  Q_{i,1}  \qquad\qquad  (0 \leq i \leq d).
\end{align*}

\begin{lemma}   {\rm (See \cite[p.\  193]{BI}.)}
\label{lem:Qij}   \samepage
\ifDRAFT {\rm lem:Qij}. \fi
We have
\begin{align*}
   Q_{i,j} &= m_j u^*_j (\th^*_i)     \qquad\qquad  (0 \leq i,j \leq d). 
\end{align*}
\end{lemma}

\begin{prop}  {\rm (See \cite[p.\  262]{BI}.) }
\label{prop:AW} \samepage
\ifDRAFT {\rm prop:AW}. \fi
Assume that $\mathcal  X$ is $P$-polynomial.
Further assume that there exists a 
$Q$-polynomial ordering $\{E_i\}_{i=0}^d$ of the primitive idempotents of $\mathcal X$.
Then
\begin{align}
   u_i (\th_j) = u^*_j (\th^*_i)
     \qquad\qquad       (0 \leq i,j \leq d).   \label{eq:AW}
\end{align}
\end{prop}

The equation \eqref{eq:AW} is called the Askey-Wilson duality.

\section{Proof of Theorems \ref{thm:main2} and \ref{thm:main4}}
\label{sec:proof}

In this section we prove Theorems \ref{thm:main2} and \ref{thm:main4}.
We will  use the notation from Section \ref{sec:duality}.

\begin{proofof}{Theorem \ref{thm:main2}}
(i) $\Rightarrow$ (ii)
Clear.

(ii) $\Rightarrow$ (i)
The ordering $\{E_i\}_{i=0}^d$ is $Q$-polynomial, since $\mathcal X$ is $P$-polynomial and $p^h_{i,j} = q^h_{i,j}$
$(0 \leq h,i,j \leq d)$.
Setting $i=1$ in \eqref{eq:AW} we obtain
\begin{equation}
   \th_j / k_1 = u^*_j (\th^*_1)  \qquad\qquad (0 \leq j \leq d).   \label{eq:poofaux2}
\end{equation}
Setting $j=1$ in \eqref{eq:AW} we obtain
\begin{equation}
   u_i(\th_1) = \th^*_i/m_1  \qquad\qquad (0 \leq i \leq d).   \label{eq:poofaux3}
\end{equation}
Setting $i=j=1$ in \eqref{eq:AW} we obtain
\begin{equation}
   \th_1 / k_1 = \th^*_1/m_1 .   \label{eq:poofaux4}
\end{equation}
We have $p^h_{i,j} = q^h_{i,j}$ for $0 \leq h,i,j \leq d$.
Consequently $u_i = u^*_i$ for $0 \leq i \leq d$, and also $k_i = m_i$ for $0 \leq i \leq d$.
By \eqref{eq:poofaux4} we find that $\th_1 = \th^*_1$.
Comparing \eqref{eq:poofaux2} and \eqref{eq:poofaux3} using these comments,
we obtain
$\th_i = \th^*_i$ for $0 \leq i \leq d$.
By this and Lemmas \ref{lem:Pij}, \ref{lem:Qij}
we obtain $P = Q$.
We have shown that $\mathcal X$ is formally self-dual wih respect to
the ordering $\{E_i\}_{i=0}^d$.
\end{proofof}

\begin{proofof}{Theorem \ref{thm:main4}}
(i) $\Rightarrow$ (ii)
Clear.

(ii) $\Rightarrow$ (i)
Note that  $\mathcal X$ is $P$-polynomial, since the ordering $\{E_i\}_{i=0}^d$ is $Q$-polynomial and $p^h_{i,j} = q^h_{i,j}$
$(0 \leq h,i,j \leq d)$.
By this and Theorem \ref{thm:main2} we find that $\mathcal X$ is formally self-dual.
\end{proofof}

{
\small

\bigskip\bigskip\noindent
Kazumasa Nomura\\
Institute of Science Tokyo\\
Kohnodai, Ichikawa, 272-0827 Japan\\
email: knomura@pop11.odn.ne.jp

\bigskip\noindent
Paul Terwilliger\\
Department of Mathematics\\
University of Wisconsin\\
480 Lincoln Drive\\ 
Madison, Wisconsin, 53706 USA\\
email: terwilli@math.wisc.edu

\bigskip\noindent
{\bf Keywords.}
association scheme;
self-dual;
$P$-polynomial;
$Q$-polynomial.
\\
\noindent
{\bf 2020 Mathematics Subject Classification.} 
05E30,  15A21, 15B10.

\end{document}